\documentclass[12pt]{amsart}
\usepackage{ifthen,color,xypic,amssymb}

\newcommand{\catD}{{D}}
\newcommand{\dbc}[1]{{D}^b(#1)}

\newcommand{\cdbc}[1]{{D}^b_c(#1)}

\newcommand{\fmf}[3]{{\Phi^{#1}_{{\scriptscriptstyle #2\!\rightarrow\! #3}}}}

\newcommand{\Hom}{{\operatorname{Hom}}}
\newcommand{\Kos}{{\operatorname{Kos}^{\scriptscriptstyle\bullet}}}
\newcommand{\Ext}{{\operatorname{Ext}}}

\newcommand{\Tor}{{\operatorname{Tor}}}

\DeclareMathOperator{\Spec}{{Spec}}

\DeclareMathOperator{\depth}{{depth}}
\DeclareMathOperator{\codepth}{{codepth}}
\DeclareMathOperator{\codim}{{codim}}

\DeclareMathOperator{\supp}{{supp}}

\newcommand{\cF}{{\mathcal F}}

\newcommand{\calH}{{\mathcal H}}

\newcommand{\cK}{{\mathcal K}}
\newcommand{\cM}{{\mathcal M}}
\newcommand{\cN}{{\mathcal N}}
\newcommand{\cO}{{\mathcal O}}

\newcommand{\bR}{{\mathbf R}}

\newcommand{\cplx}[1]{{{\mathcal #1}^{\scriptscriptstyle\bullet}}}

\newcommand{\marginnote}[1]{\ifthenelse{\isodd{\thepage}}{\normalmarginpar}
{\reversemarginpar}\marginpar{\fbox{\parbox{24mm}{\sloppy\footnotesize
#1}}}}

\DeclareMathSymbol{\functor}{\mathbin}{AMSa}{"20}

 \newtheorem{thm}{Theorem}[section]
 \newtheorem*{thm*}{Theorem}
 \newtheorem{cor}[thm]{Corollary}
 \newtheorem{lem}[thm]{Lemma}
 \newtheorem{prop}[thm]{Proposition}
 
 \theoremstyle{definition}
 \newtheorem{defin}[thm]{Definition}
 \newenvironment{defn}{\begin{defin}}{\hfill\hspace{1pt}$\triangle$\end{defin}}
 \theoremstyle{remark}
 \newtheorem{rema}[thm]{Remark}
 \newtheorem{exe}[thm]{Example}

\numberwithin{equation}{section}

\begin{document}
\title{KOSZUL COMPLEXES AND FULLY FAITHFUL INTEGRAL FUNCTORS}

\author[Fernando Sancho de Salas]{Fernando Sancho de Salas}
\email{fsancho@usal.es}
\address{Departamento de Matem\'aticas, Universidad de Salamanca, Plaza
de la Merced 1-4, 37008 Salamanca, Spain}
\date{\today}
\thanks {Work supported by research projects MTM2006-04779 (MEC)
and SA001A07 (JCYL)} \subjclass[2000]{Primary: 18E30; Secondary:
14F05, 14J27, 14E30, 13D22, 14M05} \keywords{Geometric integral
functors, Fourier-Mukai,  fully faithful, equivalence of
categories}
\begin{abstract} We characterise those objects in the derived category of a
scheme which are a sheaf supported on a closed subscheme in terms
of Koszul complexes. This is applied to generalize  to arbitrary
schemes the fully faithfullness criteria of an integral functor.

\end{abstract}
\maketitle

{\small \tableofcontents }

\section*{Introduction} Let $X,Y$ be two proper schemes over a field $k$ and let $$\Phi\colon \cdbc{X}\to \cdbc{Y}$$
be an integral functor between their derived categories of
complexes of quasi-coherent modules with bounded and coherent
cohomology. Let $\cplx{K}\in \cdbc{X\times Y}$ be the kernel of
$\Phi$. We want to characterise those kernels $\cplx{K}$ such that
$\Phi$ is a fully faithful. This was solved in \cite{BO95} for
smooth projective schemes over a field of zero characteristic. For
Gorenstein schemes and zero characteristic it was solved in
\cite{HLS05}. For Cohen-Macaulay schemes and arbitrary
characteristic it was solved in \cite{HLS??}. Here we remove the
Cohen-Macaulay hypothesis and reproduce the fully faithfullness
criteria of \cite{HLS??} for arbitrary schemes. The point is to
replace the locally complete intersection zero-cycles of
\cite{HLS??} by Koszul complexes associated to a system of
parameters. These Koszul complexes allow to characterise, for an
arbitrary scheme $X$, those objects in $\cdbc{X}$ consisting of a
sheaf supported on a closed subscheme (Propositions
\ref{1:support} and  \ref{1:support2}). This is the main
ingredient for the fully faithfullness criteria.

\medskip


\subsection*{Acknowledgements} I would like to thank Leovigildo
Alonso, who suggested to me the use of Koszul complexes to deal
with the general (non Cohen-Macaulay) case.

\section{Koszul complexes, depth and support}

We introduce Koszul complexes and use them to characterize those
objects in the derived category consisting of a sheaf supported on
a closed subscheme.

\subsection{System of parameters. Koszul complex}. Let $\cO$ be a noetherian local ring of dimension $n$
and maximal ideal $\mathfrak m$. Let $x$ be the closed point.

\begin{defn} A sequence $f = \{ f_1,\dots, f_n \}$ of $n$ elements in
$\mathfrak m$ is called a system of parameters of $\cO$ if
$\cO/(f_1,\dots, f_n)$ is a zero dimensional ring. In other words,
$(f_1,\dots,f_n)$ is a $\mathfrak m$-primary ideal. We shall also
denote $\cO/f=\cO/(f_1,\dots, f_n)$.\end{defn}

It is a basic fact of dimension theory that there always exists a
system of parameters. In fact, for any $\mathfrak m$-primary ideal
$I$, there exist $f_1,\dots,f_n$ in $I$ which are a system of
parameters of $\cO$.

We shall denote by $\Kos (f)$ the Koszul complex associated to a
system of parameters $f$. That is, if we denote $L=\cO^{\oplus n}$
and $\omega\colon L\to\cO$ the morphism given by $f_1,\dots,f_n$,
then the  Koszul complex is $\bigwedge^i_\cO L$ in degree $-i$ and
the differential $\bigwedge^i_\cO L\to \bigwedge^{i-1}_\cO L$ is
the inner contraction with $\omega$. It is immediate to see that
$\Hom^{\scriptscriptstyle\bullet}(\Kos(f),\cO)\simeq\Kos(f)[-n]$.

The cohomology modules $H^i(\Kos(f))$ are supported at $x$ (indeed
they are annihilated by $(f_1,\dots,f_n)$). Moreover $H^0(\Kos
(f))=\cO/f$ and $H^i(\Kos (f))=0$ for $i>0$ and $i<-n$.

For any complex  $\cplx{M}$ of $\cO$-modules, we shall denote
\[\aligned \Tor_i^\cO
(\Kos(f),\cplx{M})&=H^{-i}(\Kos(f)\underset{\cO}\otimes
\cplx {M})\\
\Ext_{\cO}^i
(\Kos(f),\cplx{M})&=H^i(\Hom_\cO^{\scriptscriptstyle\bullet}(\Kos(f),
\cplx {M}))
\endaligned\] From the isomorphism
$\Hom^{\scriptscriptstyle\bullet}(\Kos(f),\cO)\simeq\Kos(f)[-n]$
it follows easily that
\begin{equation}\label{ext-tor}\Ext_{\cO}^i
(\Kos(f),\cplx{M})\simeq \Tor_{n-i}^\cO
(\Kos(f),\cplx{M}).\end{equation}

\subsection{Depth. Singularity set}

The depth of an $\cO$-module $M$, $\depth(M)$, is the first
integer $i$ such that either:
\begin{itemize}
\item $\Ext^i(\cO/\mathfrak{m} ,M)\neq 0$ or
\item $H^i_x(\Spec\cO,M)\neq 0$ or
\item $\Ext_{\cO}^i(N,M)\neq 0$ for some non zero finite $\cO$-module $N$ supported at
$x$ or
\item $\Ext_{\cO}^i(N,M)\neq 0$ for any  non zero finite $\cO$-module $N$ supported at $x$.
\end{itemize}

\begin{lem} The depth of $M$ is the first integer
$i$ such that either:
\begin{itemize}
\item $\Ext_{\cO}^i(\Kos(f),M)\neq 0$ for some system of parameters $f $ of $\cO$ or
\item $\Ext_{\cO}^i(\Kos(f),M)\neq 0$ for every system of parameters $f $ of
$\cO$.
\end{itemize}
\end{lem}
\begin{proof} It is an easy consequence of the spectral sequence
\[ E_2^{p,q}=\Ext^p(H^{-q}(\Kos(f), M)\implies
E_\infty^{p+q}=\Ext^{p+q}(\Kos(f),M)\]

Indeed, let $d=\depth (M)$, $f $ a system of parameters
 of $\cO$  and  $r$  the first
integer such that $\Ext_{\cO}^i(\Kos(f),M)\neq 0$. Let us see that
$d=r$. Since $\Ext^d(H^0(\Kos(f)),M)\neq 0$, one obtains, by the
spectral sequence, that $\Ext_{\cO}^d(\Kos(f),M)\neq 0$. Hence
$d\geq r$.
Assume that $r\neq d$. Then $\Hom^{r-i}(H^{-i}(\Kos(f),M)=0$ for
any $i\geq 0$, because $H^{-i}(\Kos(f)$ is supported at $x$ and
$r-i <d$. From the exact triangles
\[ \Kos(f)_{\leq -i-1}\to \Kos(f)_{\leq -i }\to H^{-i}(\Kos(f)[i]\]
and taking into account that $\Hom^r(\Kos(f)_{\leq 0},M)=
\Hom^r(\Kos(f),M) \neq 0$ one obtains that $\Hom^r(\Kos(f)_{\leq
-i},M)\neq 0$ for any $i\geq 0$. This is absurd because
$\Kos(f)_{\leq -i}=0$ for $i>>0$.

\end{proof}

Let $\cF$ be a coherent sheaf on a scheme $X$ of dimension $n$. We
write $n_x$ for the dimension of the local ring $\cO_x$ of $X$ at
a point $x\in X$, $\cF_x$ for the stalk of $\cF$ at $x$ and $\bold
{k}(x)$ for the residual field of $x$. $\cF_x$ is a
$\cO_x$-module. The integer number
$\codepth(\cF_x)=n_x-\depth(\cF_x)$ is called the codepth of $\cF$
at $x$. For any integer $m\in\mathbb{Z}$, the \emph{$m$-th
singularity set} of $\cF$ is defined to be
$$
S_m(\cF)= \{x\in X
\;|\; \codepth(\cF_x)\geq n-m\}\, .
$$
Then, if $X$ is equidimensional, a closed point $x$ is in
$S_m(\cF)$ if and only if $\depth(\cF_x)\leq m$.

Since $\depth(\cF_x)$ is the first integer $i$ such that either
\begin{itemize}
\item $\Ext_{\cO_x}^i(\bold{k}(x),\cF_x)\neq 0$ or
\item $H^i_x(\cF_x)\neq 0$ or
\item $\Ext_{\cO_x}^i(\Kos(f_x),\cF_x)\neq 0$ for some system of parameters $f_x$ of $\cO_x$ or
\item $\Ext_{\cO_x}^i(\Kos(f_x),\cF_x)\neq 0$ for every system of parameters $f_x$ of $\cO_x$
\end{itemize}
we have alternative descriptions of $S_m(\cF)$:
\begin{equation} \label{e:descrip}
\begin{aligned}
S_m(\cF) &= \{x\in X
\;|\; H^i_x(\Spec\cO_{X,x},\cF_x)\neq 0 \text{ for some } i\leq m+n_x-n\} \\
& =  \{x\in X
\;|\; \Ext_{\cO_x}^i(\Kos(f_x),\cF_x)\neq 0 \text{ for some } i\leq m+n_x-n  \\
& \qquad \text{and some system of parameters $f_x$ of $\cO_{X,x}$} \} \\
& =  \{x\in X
\;|\; \Ext_{\cO_x}^i(\Kos(f_x),\cF_x)\neq 0 \text{ for some } i\leq m+n_x-n  \\
& \qquad \text{and any system of parameters $f_x$ of
$\cO_{X,x}$}\}
\end{aligned}
\end{equation}

\begin{lem} \label{l:sm}\cite[Lemma 1.10]{HLS05}. If $X$ is smooth, then the $m$-th singularity set of $\cF$
can be described as
$$
 S_m(\cF)=\cup_{p\ge n-m} \{ x\in X \;|\; \Tor_p^{\cO_x}(\bold{k} (x), \cF_x)\neq
0\}\,,
$$
where $\bold{k} (x)$ is the residue field of $\cO_x$.
\end{lem}

In the singular case, this characterization of $S_m(\cF)$ is not
true. There is a similar interpretation for Cohen-Macaulay schemes
replacing $\bold{k} (x)$ by $\cO_{Z_x}$ where $Z_x$ is a locally
complete intersection zero cycle supported on $x$ (see \cite[Lemma
3.5]{ HLS??}). Now, for arbitrary schemes, the analogous
interpretation is the following.

\begin{lem}  \label{l:SetSing} The $m$-th singularity set $S_m(\cF)$ can be
described as
\begin{align*}
S_m(\cF) &= \{x\in X \;|\; \text{there is an integer }  i\geq
n-m\text{ with } \Tor_i^{\cO_x}(\Kos(f_x), \cF)\neq 0  \\
& \qquad \text{for any system of parameters $f_x$ of $\cO_{X,x}$}
\}\,.
\end{align*}
\end{lem}
\begin{proof} It follows from \eqref{ext-tor}  and \eqref{e:descrip}.
\end{proof}


\begin{prop}\label{p:Sm} \cite[Prop 1.13]{HLS05}. Let $X$ be an equidimensional scheme of dimension $n$
and $\cF$ a coherent sheaf on $X$. \begin{enumerate} \item
$S_m(\cF)$ is a closed subscheme of $X$ and $\codim S_m(\cF)\geq
n-m$. \item If $Z$ is an irreducible component of the support of
$\cF$ and $c$ is the codimension of $Z$ in $X$, then $\codim
S_{n-c}(\cF)=c$ and $Z$ is also an irreducible component of
$S_{n-c}(\cF)$.
 \end{enumerate}
\end{prop}

\begin{cor}\label{lemafer}\cite[Cor. 1.14]{HLS05}. Let $X$ be a  scheme  and let $\cF$ be a
coherent $\cO_X$-module.  Let $h\colon Y\hookrightarrow X$ be an
irreducible component of the support of $\cF$ and $c$ the
codimension of $Y$ in $X$. There is a non-empty open subset $U$ of
$Y$ such that for any $x\in U$ and any system of parameters $f_x$
of $\cO_{X,x}$  one has
\begin{align*}
\Tor_c ^{\cO_x}(\Kos(f_x), \cF_x)&\neq 0
\\
\Tor_{c+i}^{\cO_x}(\Kos(f_x), \cF_x)&= 0\,, \quad \text{ for every
$i>0$.}
\end{align*}
\end{cor}

\begin{proof} By Lemma \ref{l:SetSing} the locus of the points that verify the conditions is
$U=Y\cap(S_{n-c}(\cF)-S_{n-c-1}(\cF))$, which is open in $Y$ by
Proposition \ref{p:Sm}. Proving that $U$ is not empty is a local
question, and we can then assume that $Y$ is the support of $\cF$.
Now $Y=S_{n-c}(\cF)$ by (2) of Proposition \ref{p:Sm} and
$U=S_{n-c}(\cF)-S_{n-c-1}(\cF)$ is non-empty because
 the codimension of $S_{n-c-1}(\cF)$ in $X$ is greater or equal than $c+1$ again by
Proposition \ref{p:Sm}.
\end{proof}

For any scheme $X$ we denote by $\catD(X)$ the derived category of
complexes of quasi-coherent $\cO_X$-modules and by $\cdbc{X}$ the
faithful subcategory consisting of those complexes with bounded
and coherent cohomology sheaves.

The following proposition characterises objects of the derived
category supported on a closed subscheme.
\begin{prop} \emph{\cite[Prop.~1.5]{BO95}\cite[Prop. 1.15]{HLS05}}. Let $j\colon Y\hookrightarrow X$ be a closed
immersion
of codimension $d$ of irreducible  schemes and $\cplx{K}$ an
object of $\cdbc X$. Assume that
\begin{enumerate}
\item If $x\in X-Y$ is a closed point, then there exists a system of parameters $f_x$ of $\cO_x$ such that
 $\Tor_i^{\cO_x}(\Kos(f_x), \cplx{K}_x)=0$ for every $i$.
\item If $x\in Y$ is a closed point, then there exists a system of parameters $f_x$ of $\cO_x$ such that
$\Tor_i^{\cO_x}(\Kos(f_x), \cplx{K}_x)=0$  when either $i<0$ or
$i>d$.
\end{enumerate}
Then there is a sheaf $\cK$ on $X$ whose topological support is
contained in $Y$ and such that $\cplx{K}\simeq\cK$ in $\cdbc X$.
Moreover, this topological support coincides with $Y$ unless
$\cplx{K}=0$.
 \label{1:support}
\end{prop}
\begin{proof} We just reproduce the proof of \cite[Prop.
1.15]{HLS05}, with the corresponding changes. Let us write
$\calH^q=\calH^q(\cplx K)$. For every system of parameters $f_x$
of $\cO_x$  there is a spectral sequence
$$
E_2^{-p,q}=\Tor_p^{\cO_x} ( \Kos(f_x), \calH^q_x )\implies
E_\infty^{-p+q}=\Tor_{p-q}^{\cO_x}( \Kos(f_x),  \cplx{K}_x)
$$
Let $q_0$ be the maximum of the $q$'s with $\calH^q\neq 0$. If
$x\in \supp(\calH^{q_0})$, one has that $\Tor_0^{\cO_x}(\Kos(f_x),
\calH^{q_0}_x)\simeq H^0(\Kos(f_x))\otimes_{\cO_x} \calH^{q_0}_x
\neq 0$ for every system of parameters $f_x$ of $\cO_x$. A nonzero
element in $\Tor_0^{\cO_x}(\Kos(f_x), \calH^{q_0}_x)$ survives up
to infinity in the spectral sequence. Since there is a system of
parameters $f_x$ of $\cO_x$ such that
$E_\infty^{q}=\Tor_{-q}^{\cO_x}(\Kos(f_x), \cplx{K})=0$ for every
$q>0$ by hypothesis, one has $q_0\le 0$. A similar argument shows
that the topological support of all the sheaves $\calH^q$ is
contained in $Y$: assume that this is not true and let us consider
the maximum $q_1$ of the $q$'s such that $\calH^q_x\neq 0$ for a
certain point $x\in X-Y$; then $\Tor_0^{\cO_x}(\Kos(f_x),
\calH^{q_1}_x)\neq 0$ and a nonzero element in
$\Tor_0^{\cO_x}(\Kos(f_x), \calH^{q_1}_x)$ survives up to infinity
in the spectral sequence, which is impossible since
$\Tor_i^{\cO_x}(\Kos(f_x), \cplx K)=0$ for every $i$.

Let  $q_2\le q_0$ be the minimum of the $q$'s with $\calH^q\neq
0$. We know that $\calH^{q_2}$ is topologically supported on a
closed subset of $Y$. Take a component $Y'\subseteq Y$ of the
support. If $c\ge d$ is the codimension of $Y'$, then there is a
non-empty open subset $U$ of $Y'$ such that
$\Tor_c^{\cO_x}(\Kos(f_x),\calH^{q_2}_x)\neq 0$ for any closed
point $x\in U$ and any system of parameters $f_x$ of $\cO_x$, by
Corollary \ref{lemafer}. Elements in
$\Tor_c^{\cO_x}(\Kos(f_x),\calH^{q_2}_x)$ would be killed in the
spectral sequence by $\Tor_{p}^{\cO_x}(\Kos(f_x),
\calH^{q_2+1}_x)$ with $p\ge c+2$. By Lemma \ref{l:SetSing}  the
set
 $$ \{x\in X \;|\;
\Tor_i^{\cO_x}(\Kos(f_x), \calH^{q_2+1}_x)\neq 0 \text{ for some }
i\geq c+2 \text{ and any parameters $f_x$ of $\cO_x$} \}
$$ is equal to $S_{n-(c+2)}(\calH^{q_2+1})$ and then has
codimension greater or equal than $c+2$ by Proposition \ref{p:Sm}.
Thus there is a point $x\in Y'$ such that any nonzero element in
$\Tor_c^{\cO_x}(\Kos(f_x),\calH^{q_2}_x)$ survives up to the
infinity in the spectral sequence. Therefore,
$\Tor_{c-q_2}^{\cO_x}(\Kos(f_x), \cplx{K}_x)\neq 0$ for any system
of parameters $f_x$ of $\cO_x$. Thus $c-q_2\le d$ which leads to
$q_2\ge c-d\ge 0$ and then $q_2=q_0=0$. So $\cplx{K}= \calH^0$ in
$\dbc X$ and the topological support of $\cK=\calH^0$ is contained
in $Y$. Actually, if $\cplx{K}\neq 0$, then this support is the
whole of $Y$: if this was not true, since $Y$ is irreducible, the
support would have a component $Y'\subset Y$ of codimension $c>d$
and one could find, reasoning as above, a non-empty subset $U$ of
$Y'$ such that $\Tor_c^{\cO_x}(\Kos(f_x),\cplx{K}_x)\neq 0$ for
all $x\in U$ and all system of parameters $f_x$ of $\cO_x$. This
would imply that $c\le d$, which is impossible.
\end{proof}

Assume now that $X$ is separated. Let $x$ be a closed point of $X$
and $\phi_x\colon\Spec\cO_x\to X$ the natural morphism. Let $f_x$
be a system of parameters of $\cO_x$. We shall still denote by
$\Kos(f_x)$ the direct image by $\phi_x$ of the Koszul complex
$\Kos(f_x)$. Let $U$ be an affine open subset containing $x$. Then
$\phi_x$ is the composition of $\phi'_x \colon\Spec\cO_x\to U$
with the open embedding $i_U\colon U\hookrightarrow X$. Since $X$
is separated, $i_U$ is an affine morphism, and then
${\phi_x}_\ast\simeq\bR {\phi_x}_\ast$.

One has
that
\begin{lem} For any $\cplx{K}\in \catD(X)$ one has
\[ \Hom^i_{\catD(X)}(\Kos(f_x),\cplx{K})\simeq
\Ext^i_{\cO_x}(\Kos(f_x),\cplx{K}_x)\]
\end{lem}

\begin{proof} Let $C$ be the cone of $\cplx{K}\to
{\phi_x}_\ast \phi_x^*\cplx{K}$. It is clear that $x\notin \supp
(C)$. On the other hand ${\phi_x}_*\Kos(f_x)$ is supported at $x$.
Then $\Hom^i({\phi_x}_*\Kos(f_x), C)=0$ and
\[ \Hom^i_{\catD(X)}({\phi_x}_* \Kos(f_x),\cplx{K}) \simeq \Hom^i_{\catD(X)}({\phi_x}_*
\Kos(f_x),{\phi_x}_\ast \phi_x^*\cplx{K})\] and one concludes
because $\phi_x^\ast {\phi_x}_* \Kos(f_x) \simeq \Kos(f_x)$.
\end{proof}

Taking into account the equation \eqref{ext-tor}, Proposition
\ref{1:support} may be reformulated as follows:
\begin{prop}  \label{1:support2} Let $j\colon Y\hookrightarrow X$ be a closed immersion
of codimension $d$ of irreducible  schemes of dimensions $m$ and
$n$ respectively, and let $\cplx{K}$ be an object of $\cdbc X$.
Assume that for any closed point $x\in X$ there is a system of
parameters $f_x$ of $\cO_x$ such that
$$
 \Hom^i_{\catD(X)}(\Kos(f_x),\cplx{K})= 0\,,
$$
unless $x\in Y$ and $m\le i\le n$. Then there is a sheaf $\cK$ on
$X$ whose topological support is contained in $Y$ and such that
$\cplx{K}\simeq\cK$ in $\cdbc X$. Moreover, the topological
support is $Y$ unless $\cplx{K}=0$. \qed\end{prop}

\subsubsection{Spanning classes}

\begin{lem} \label{l:spanning} For each closed point $x\in X$ choose a system of parameters $f_x$ of $\cO_x$.
The set
$$\Omega=\{ \Kos(f_x) \text{ for all closed points } x\in X \}$$ is a spanning class
for $\cdbc{X}$.
\end{lem}
\begin{proof} Take a non-zero object $\cplx E$ in $\cdbc{X}$.
Let $q_0$ be the maximum of  the $q$'s such that $\calH^{q}(\cplx
E)\neq 0$, $x$  a closed point of the support of $\calH^{q}(\cplx
E)$ and $-l$ the minimum of the $p$'s such that
$H^p(\Kos(f_x))\neq 0$. Then
$$\aligned \Hom^{-l-q_0}_{\catD (X)}(\cplx E,\Kos(f_x))& \simeq
\Hom_{\cO_X}(H^{q_0}(\cplx{E}),H^{-l}(\Kos(f_x))\\ &\simeq
\Hom_{\cO_x}(H^{q_0}(\cplx{E})_x,H^{-l}(\Kos(f_x)) \neq
0.\endaligned$$

On the other hand, by Proposition \ref{1:support2} with
$Y=\emptyset$, if $\Hom^{i}_{\catD (X)}(\Kos(f_x),\cplx E)= 0$ for
every $i$ and every $x$, then $\cplx{E}=0$.
\end{proof}

\section{Fully faithful Integral functors}\label{ss:basechange}

In this section scheme means a separated scheme of finite type
over an algebraically closed field $k$.

Let $X$ and $Y$ be proper schemes, $\cplx{K}$ an object in
$\cdbc{X\times Y}$ and
\[ \fmf{\cplx{K}}{X}{Y}\colon\catD(X)\to\catD(Y)\] the integral functor associated to
$\cplx{K}$. If $X$ is projective and $\cplx{K}$ has finite
homological dimension over both $X$ and $Y$, then
$\fmf{\cplx{K}}{X}{Y}$ maps $\cdbc{X}$ to $\cdbc{Y}$ and it has an
integral right adjoint (see \cite[Def. 2.1]{HLS??}, \cite[Prop.
2.7]{HLS??} and \cite[Prop. 2.9]{HLS??}).

The notion of strong simplicity is the following.

\begin{defn} An object  $\cplx{K}$ in $\cdbc{X\times Y}$ is \emph{strongly simple} over $X$ if it satisfies the
following conditions:
\begin{enumerate}
\item For every  closed point $x\in X$ there is a system of parameters $f_x$ of $\cO_x$  such that
$$
\Hom^i_{\catD(Y)}(\fmf{\cplx{K}}{X}{Y}(\Kos(f_{x_1}),\fmf{\cplx{K}}{X}{Y}(\bold{k}(x_2)))=0
$$
unless $x_1= x_2$ and $0\leq i\leq \dim X$.
\item$\Hom^0_{\catD(Y)}(\fmf{\cplx{K}}{X}{Y}(\bold{k}(x)),
\fmf{\cplx{K}}{X}{Y}(\bold{k}(x)))= k$ for every closed point
$x\in X$.
\end{enumerate}
\label{1:strgsplcplxCM}
\end{defn}

\begin{thm}\label{1:ffcritCM}  Let $X$ and $Y$ be proper schemes over
an algebraically closed field of characteristic zero, and let
$\cplx{K}$ be an object in $\cdbc{X\times Y}$ of finite
homological dimension over both $X$ and $Y$. Assume also that $X$
is projective and integral. Then the functor
$\fmf{\cplx{K}}{X}{Y}\colon \cdbc{X}\to \cdbc{Y}$ is fully
faithful if and only if the kernel $\cplx{K}$ is strongly simple
over $X$.
\end{thm}

\begin{proof} The same proof as \cite[Thm. 3.6]{HLS??} works,
replacing  the use of Proposition 3.1 of \cite{HLS??} by its
analogous result (Proposition \ref{1:support2}).
\end{proof}

\begin{defn} \label{d:ortho}
An object $\cplx K$ of $\cdbc{X\times Y}$ satisfies the
 orthonormality conditions over $X$ if it has the
following properties:
\begin{enumerate}
\item For every  closed point $x\in X$ there is a system of parameters $f_x$ of $\cO_x$  such that
$$
\Hom^i_{\catD(Y)}(\fmf{\cplx{K}}{X}{Y}(\Kos(f_{x_1}),\fmf{\cplx{K}}{X}{Y}(\bold{k}(x_2)))=0
$$
unless $x_1= x_2$ and $0\leq i\leq \dim X$.
\item There exists a closed point $x$ such that at least one of the following conditions is fulfilled:
\begin{enumerate} \item[(2.1)] $\Hom^0_{\catD(Y)}(\fmf{\cplx{K}}{X}{Y}(\cO_X),
\fmf{\cplx{K}}{X}{Y}(\bold{k}(x)))\simeq k$.
\item[(2.2)] $\Hom^0_{\catD(Y)}(\fmf{\cplx{K}}{X}{Y}(\Kos(f_x)),
\fmf{\cplx{K}}{X}{Y}(\bold{k}(x)))\simeq k$ for any system of
parameters $f_x$ of $\cO_x$.
\item[(2.2$^*$)] $\Hom^0_{\catD(Y)}(\fmf{\cplx{K}}{X}{Y}(\cO_x/f_x),
\fmf{\cplx{K}}{X}{Y}(\bold{k}(x)))\simeq k$ for any system of
parameters $f_x$ of $\cO_x$.
\item[(2.3)] $1\leq \dim  \Hom^0_{\catD(Y)}(\fmf{\cplx{K}}{X}{Y}(\Kos(f_x)),
\fmf{\cplx{K}}{X}{Y}(\cO_x/f_x))\leq l(\cO_x/f_x)$ for any system
of parameters $f_x$ of $\cO_x$, where $l(\cO_x/f_x)$ is the length
of $\cO_x/f_x$.
\item[(2.3$^*$)] $1\leq \dim  \Hom^0_{\catD(Y)}(\fmf{\cplx{K}}{X}{Y}(\cO_x/f_x),
\fmf{\cplx{K}}{X}{Y}(\cO_x/f_x))\leq l(\cO_x/f_x)$ for any system
of parameters $f_x$ of $\cO_x$.
\end{enumerate}
\end{enumerate}
\end{defn}

\begin{thm}\label{1:ffcritCMp}  Let $X$ and $Y$ be proper schemes over
an algebraically closed field of arbitrary characteristic, and let
$\cplx{K}$ be an object in $\cdbc{X\times Y}$ of finite
homological dimension over both $X$ and $Y$. Assume also that $X$
is projective, Cohen-Macaulay, equidimensional and connected. Then
the functor $\fmf{\cplx{K}}{X}{Y}\colon \cdbc{X}\to \cdbc{Y}$ is
fully faithful if and only if the kernel $\cplx{K}$ satisfy the
orthonormality conditions over $X$ (Definition \ref{d:ortho}).
\end{thm}

\begin{proof} The proof is essentially the same as \cite[Thm.
3.8]{HLS??}. We give the details.

The direct is immediate. Let us see the converse. Let us denote
$\Phi=\fmf{\cplx{K}}{X}{Y}$. One knows that $\Phi$ has a right
adjoint $H$ and that $H\circ \Phi\simeq\fmf{\cM}{X}{X}$. Using
condition (1) of Definition \ref{d:ortho}, one sees that $\cM$ is
a sheaf whose support is contained in  the diagonal and
$\pi_{1\ast}\cM$ is locally free. Since $X$ is connected, we can
consider the rank $r$ of $\pi_{1\ast}\cM$, which is nonzero by
condition (2) of Definition \ref{d:ortho}; thus the support of
$\cM$ is the diagonal. To conclude, we have only to prove that
$r=1$.

Since $\cM$ is a sheaf topologically supported on the diagonal and
$\pi_{1\ast}\cM$ is locally free, it follows that  if $\cF$ is a
sheaf, then $\fmf{\cM}{X}{X}(\cF)$ is also a sheaf.

Now assume that $\cplx{K}$ satisfies (2.1) of Definition
\ref{d:ortho}. Then
$$\Hom^0_{\catD(X)}(\cO_X,\fmf{\cM}{X}{X}(\bold{k}(x))) \simeq
\Hom^0_{\catD(Y)} (\fmf{\cplx{K}}{X}{Y}(\cO_X),
\fmf{\cplx{K}}{X}{Y}(\bold{k}(x))) \simeq k.$$ Hence
$\fmf{\cM}{X}{X}(\bold{k}(x))\simeq \bold{k}(x)$; that is,
$j_x^\ast\cM\simeq \bold{k}(x)$, where $j_x\colon
\{x\}\hookrightarrow X$ is the inclusion, and $r=1$.

If $\cplx{K}$ satisfies (2.2) of Definition \ref{d:ortho}, then
$$\aligned \Hom_{\cO_X}(\cO_x/f_x, j_x^\ast\cM)&\simeq
\Hom^0_{\catD(X)}(\Kos(f_x),j_x^\ast\cM)\\ &\simeq
\Hom^0_{\catD(X)}(\Kos(f_x),\fmf{\cM}{X}{X}(\bold{k}(x)))\\
&\simeq \Hom^0_{\catD(Y)} (\fmf{\cplx{K}}{X}{Y}(\Kos(f_x)),
\fmf{\cplx{K}}{X}{Y}(\bold{k}(x))) \simeq k\endaligned$$ for any
system of parameters $f_x$ of $\cO_x$. Hence $j_x^\ast\cM\simeq
\bold{k}(x)$ and $r=1$.

(2.2$^*$) is equivalent to (2.2), because
\[\aligned  \Hom^0_{\catD(Y)} (\fmf{\cplx{K}}{X}{Y}(\Kos(f_x)),
\fmf{\cplx{K}}{X}{Y}(\bold{k}(x)))  &\simeq
\Hom^0_{\catD(X)}(\Kos(f_x),\fmf{\cM}{X}{X}(\bold{k}(x)))\\
&\simeq
\Hom^0_{\catD(X)}(\cO_x/f_x,\fmf{\cM}{X}{X}(\bold{k}(x)))\\ &
\simeq \Hom^0_{\catD(Y)} (\fmf{\cplx{K}}{X}{Y}(\cO_x/f_x),
\fmf{\cplx{K}}{X}{Y}(\bold{k}(x)))\endaligned
\] where the second isomorphism
is due to the fact that $\fmf{\cM}{X}{X}(\bold{k}(x))$ is a sheaf
and to $H^0(\Kos(f_x))=\cO_x/f_x$.

Finally, assume that $\cplx{K}$ satisfies (2.3)  of Definition
\ref{d:ortho} (which is equivalent to (2.3$^*$) by similar
arguments), and let us prove that then condition (2.2$^*$) of
Definition \ref{d:ortho} holds as well.

We already know that if $\cF$ is a sheaf supported at a point $x$,
then $\phi(\cF)=\fmf{\cM}{X}{X}(\cF)$ is also a sheaf supported at
$x$.  Moreover $\phi$ is exact and it has a left adjoint $G^0$
(see the proof of \cite[Thm. 3.8]{HLS??}). Let us denote
$B=\cO_x/f_x$.

First notice that
$$ \Hom^0_{\catD(Y)}(\fmf{\cplx{K}}{X}{Y}(B),
\fmf{\cplx{K}}{X}{Y}(B)) \simeq \Hom_{\cO_X} (B,
\fmf{\cM}{X}{X}(B)) \simeq \Hom_{\cO_X}(G^0(B),B)
 $$  Hence, condition (2.3$^*$) means
that
\[(*) \hskip 2cm  1\leq\dim  \Hom_{\cO_X}(G^0(B),B)\leq
l(B).\hskip 2cm \]

Analogously, condition (2.2$^*$) means that $\Hom_{\cO_X} (G^0(B),
\bold{k}(x))\simeq k$.

Using the exactness of $\phi$, one proves by induction on  the
length $\ell(\cF)$ that the unit map $\cF\to \phi(\cF)$ is
injective for any sheaf $\cF$ supported on $x$.  It follows easily
(see the proof of \cite[Thm. 3.8]{HLS??} for details)  that the
morphism $G^0(\cF)\to \cF$ is an epimorphism. In particular
$\eta\colon G^0(B)\to B$ is surjective, and $\dim
\Hom_{\cO_X}(G^0(B),B)\geq \ell (B)$. By ($\ast$),  $\dim
\Hom_{\cO_X}(G^0(B),B) = \ell(B)$. Now the proof follows as in
\cite[Thm. 3.8]{HLS??}: Let $j\colon\Spec B\hookrightarrow X$ be
the inclusion. The exact sequence of $B$-modules
$$
0\to\cN\to j^*G^0(B)\xrightarrow{j^\ast(\eta)} B\to 0
$$
splits, so that
$$
 0\to  \Hom_{B}(B  ,B)\to \Hom_{B}( j^*G^0(B) ,
B)\to \Hom_{B}( \cN  ,B)\to 0
$$
is an exact sequence. Then, $\Hom_{B}( \cN ,B)=0$ because the two
first terms have the same dimension. Let us see that this implies
$\cN=0$. If $\bold{k}(x)\to B$ is a nonzero, and then injective,
morphism, we have $\Hom_{B}(\cN,\bold{k}(x))=0$ so that $\cN=0$ by
Nakayama's lemma. In conclusion, $j^*G^0(B)\simeq B$, and then
$\Hom_{\cO_X} (G^0(B), \bold{k}(x))\simeq  k$.
\end{proof}

\bibliographystyle{siam}
\bibliography{bibartORG}

\bigskip

\end{document}